\documentclass{article}
\usepackage{amsmath}
\usepackage[utf8]{inputenc}
\usepackage{amssymb}
\usepackage{xparse}
\usepackage{mathabx}
\usepackage{physics}
\usepackage{amsthm}
\usepackage{mathrsfs}
\usepackage{verbatim}
\usepackage{dsfont}
\usepackage{tikz}
\usepackage{qtree}
\usepackage{forest}
\usepackage{mathtools}
\usepackage{cancel}
\usepackage{hyperref}
\usepackage{MnSymbol}
\usepackage{authblk}
\usepackage{url}
\usepackage{adjustbox}
\usepackage{caption}
\usepackage{subcaption}
\newtheorem{theorem}{Theorem}
\newtheorem{lemma}{Lemma}
\newtheorem{cor}{Corollary}
\theoremstyle{definition}
\newtheorem{defn}{Definition}

\theoremstyle{remark}

\newcommand{\End}{\text{End}}

\newcommand{\defref}[1]{\hyperref[#1]{Definition \ref*{#1}}}
\newcommand{\Defref}[1]{\hyperref[#1]{Definition \ref*{#1}}}
\newcommand{\lemref}[1]{\hyperref[#1]{Lemma \ref*{#1}}}
\newcommand{\Lemref}[1]{\hyperref[#1]{Lemma \ref*{#1}}}
\newcommand{\thmref}[1]{\hyperref[#1]{Theorem \ref*{#1}}}
\newcommand{\Thmref}[1]{\hyperref[#1]{Theorem \ref*{#1}}}
\newcommand{\corref}[1]{\hyperref[#1]{Corollary \ref*{#1}}}
\newcommand{\Corref}[1]{\hyperref[#1]{Corollary \ref*{#1}}}

 \tikzset{
          solid node/.style={circle,draw,inner sep=1.2,fill=black},
          green node/.style={circle,draw,inner sep=1.2,fill=green, draw=green},
          red node/.style={circle,draw,inner sep=1.2,fill=red, draw=red},
          blue node/.style={circle,draw,inner sep=1.2,fill=blue, draw=blue},
          yellow node/.style={circle,draw,inner sep=1.2,fill=yellow, draw=yellow},
          cyan node/.style={circle,draw,inner sep=1.2,fill=cyan, draw=cyan},
          hollow node/.style={circle,draw,inner sep=1.2},
          big node/.style={elipse, draw=green!60, fill=green!5},
          big solid node/.style={circle,draw,inner sep=1.2,fill=black, minimum size=0.3cm},
          big green node/.style={circle,draw,inner sep=1.2,fill=green, draw=green, minimum size=0.3cm},
          big red node/.style={circle,draw,inner sep=1.2,fill=red, draw=red, minimum size=0.3cm},
          big blue node/.style={circle,draw,inner sep=1.2,fill=blue, draw=blue, minimum size=0.3cm},
          big yellow node/.style={circle,draw,inner sep=1.2,fill=yellow, draw=yellow, minimum size=0.3cm},
          big cyan node/.style={circle,draw,inner sep=1.2,fill=cyan, draw=cyan, minimum size=0.3cm},
          left label/.style={above left,midway},
          right label/.style={above right,midway}
        }

\title{Embedding Higman-Thompson groups of unfolding trees into the Leavitt path algebras}
\author{Roman Gorazd\\ \url{roman.gorazd@gmail.com}}
\affil{University of Newcastle}
\date{February 2025}

\begin{document}

    \maketitle
    \begin{abstract}
        The isomorphism problem of regular Higman-Thompson groups was solved in \cite{PARDO2011172}, via embedding it into the Leavitt algebra. In this paper, we will expand these results to embed the Higman-Thompson groups of unfolding trees of directed graphs (also known as full groups of one-sided shifts) into the Leavitt path algebra. This embedding allows us to show that any isomorphism of rooted Leavitt path algebras over $\mathbb{Z}$ induces an isomorphism between Higman-Thompson groups.
    \end{abstract}
    
    \section{Introduction}  
    In this paper, we will explore the connection between the Leavitt path algebra of a directed graph and the Higman-Thompson group of the unfolding tree of that same graph, these groups are better known as the full groups of the one-sided shifts as discussed in \cite{matui2013}. This class of groups is a generalization of Higman-Thompson groups by looking at the automorphism of the boundary of a path space of a graph (which is usually a cantor space) that are locally shifts. The classical case considered by Higman\cite{higman74} is equivalent to the graph being a source with $k$ edged pointing to a vertex with $k$ loops adjacent to it. The class is closed under direct products. This paper will expand the methodology used in \cite{PARDO2011172} in order to solve the isomorphism problem of classical Higman-Thompson groups. In that paper, embedding the original Higman-Thompson group\cite{higman74} into the Leavitt algebra was used together with results about isomorphism of Leavitt path algebras from \cite{abrams07} to show a necessary condition for two Higman-Thompson groups of quasi regular trees to be isomorphic. The converse statement, that this condition is sufficient, was already shown by Higman in \cite{higman74}. While extending the methodology of \cite{PARDO2011172} to Higman-Thompson groups of unfolding trees allows us to show that isomorphism of certain subalgebras of Leavitt path algebras is sufficient for the Higman-Thompson groups to be isomorphic, it cannot be determined the same way if that condition is necessary.
    
    We will be looking at the $*$-subalgebra of the Leavitt path algebra generated by the paths that start at a fixed root of the graph. We will call it the rooted Leavitt path algebra.
    Using the notation from \cite{scott84} we can look at cofinite spaces and their bases to describe representatives of Higman-Thompson automorphisms. Using these bases, we can also describe the unitary elements of the Leavitt path algebra, based on $\mathds{Z}$, $U(L_{\mathds{Z}}(G,R))$. This allows us to embed the Higman-Thompson groups into the unitary elements as in \lemref{ch4-lem6}, by defining a canonical Leavitt path algebra element for any Higman-Thompson representative and showing that they are equal if and only if the representatives produce the same Higman-Thompson automorphism. However, the image of this embedding does not have to be fixed under automorphisms of the Leavitt algebra (or its unitary group). In order to show that an isomorphism between rooted Leavitt path algebras induces an isomorphism between the Higman-Thompson groups, we will look at the kernel of a left inverse of this embedding. This kernel: $DU(L_{\mathds{Z}}(G,R))$ can be seen as an analogue of the diagonal group of matrices. Using the group of symmetric elements, we can describe this kernel as a subset of the rooted Leavitt path algebra that is defined by a first order sentence. So any automorphism will preserve $DU(L_{\mathds{Z}}(G,R))$ and $U(L_{\mathds{Z}}(G,R))$, and thus it also preserves their factor, which is isomorphic to the Higman-Thompson group. Thus any isomorphism between rooted Leavitt path algebras will induce an isomorphism between Higman-Thompson groups. Note that the same does not follow for isomorphisms of unitary groups, since the characterization of $DU(L_{\mathds{Z}}(G,R))$ relies on the symmetric elements.
    
    Lastly, we will apply our results to show one way of reducing a rooted graph that preserves the Higman-Thompson group. The same result can be shown more directly, however using the results of the paper shortens the proof considerably.

    \section{Definitions}
    In this paper, the graphs will be directed with multiple edges. Formally, graphs are $4$-tuples $G=(VG,EG,o,t)$, where $VG$ and $EG$ are the vertex and edge sets respectively and $o,t:EG\to VG$ are the origin and terminus functions. We will also assume that all graphs throughout this paper are locally finite, i.e. \(\forall v\in VG,\, |o^{-1}(v)|<\infty\).  
    We can define the set of paths in $G$ to be
        \[\mathcal{W}(G):=\{e_1e_2\dots e_n\in EG^*\mid \forall i<n\ e_{i-1}=e_i\}\cup \{\epsilon_v\mid v\in VG\}, \] 
    where $\epsilon_v$ denotes the empty path based on $v$. We will denote $\mathcal{W}^{+}(G)$ to be the set of non-empty paths. For any path $p\in \mathcal{W}(G)$, we denote $|p|$ to be the length of it, additionally we set $O(e_1\dots e_n):=o(e_1),\ T(e_1\dots e_n):=t(e_n)$ to be the origin and terminus, respectively of $p=e_1\dots e_n$ (with $O(\epsilon_v):=v=:T(\epsilon_v)$ for any vertex $v$). We will write $\mathcal{W}(G,v)$ to be the set of all paths originating in $v$. For any graph $G$ and any vertex $v$, we will denote by $G_v$ the subgraph consisting of the vertices and edges that can be reached by a directed path starting at $v$ (and all edges that start and end at those vertices). A vertex $R$ such that $G_R=G$, is called a root. 
    For any such root, we can define the unfolding tree based on $R$ by setting
        \begin{align*}
            V\mathcal{T}(G,R):&=\mathcal{W}(G,R)\\
            E\mathcal{T}(G,R):&=\{ (p,pe)\in \mathcal{W}(G)^2\mid O(p)=O(pe)=R, e\in EG \}\\
            \forall (p,pe)\in E\mathcal{T}(G,S),\ &o_{\mathcal{T}(G,R)}( (p,pe) ):=p\text{ and } t_{\mathcal{T}(G,R)}( (p,pe) ):=pe
        \end{align*}
    We will also sort the vertices by the prefix order $\preceq$ (i.e. $p\preceq q\iff \exists r\in\mathcal{W}(G),\ q=pr$), this ordering makes $\mathcal{T}(G,R)$ into a meet-semilattice. As in \cite{scott84} we can define cofinite subspaces of $\mathcal{T}(G,R)$ to be subsets $\mathcal{S}\subseteq\mathcal{T}(G,R)$ such that:
        \begin{itemize}
            \item $|V\mathcal{T}(G,R)\setminus \mathcal{S}|<\infty$
            \item $\forall p\in \mathcal{S},\ q\in\ \mathcal{T}(G, R)\ (p\preceq q)\implies q\in \mathcal{S}$.
        \end{itemize}
    If we define any set of paths to be independent if none of the elements are prefixes of each other, we can see that any cofinite subspace has a unique finite independent subset $B\subseteq \mathcal{S}$, such that
        \[\forall p\in \mathcal{S},\exists b\in B,\ b\preceq p.\]
    We will call such a set the basis of $\mathcal{S}$. Conversely, for any inclusion-maximal finite independent set $B$, there is a unique cofinite subspace $\mathcal{S}(B)$, that has $B$ as a basis (consisting of all paths that have a prefix in $B$), so any inclusion-maximal independent finite set will be referred to as a basis.
    We will call an isomorphism between two cofinite subspaces $\phi:\mathcal{S}\to \mathcal{S}'$ an almost automorphism representative. We will call a representative, a Higman-Thompson representative if for the basis $B$ of its domain we have 
        \[\forall b\in  B,\forall p\in \mathcal{W}(G,T(b)),\ \phi(bp)=\phi(b)p.\]
    We will say that two representatives $\phi,\psi$ are equivalent if there is a cofinite subspace $\mathcal{S}$ such that $\phi|_{\mathcal{S}}=\psi|_{\mathcal{S}}$. The group of almost automorphisms consists of the equivalence classes of almost automorphism representatives (with the operation being the composition of compatible representatives). The Higman-Thompson group will be the subgroup of all equivalence classes that have a Higman-Thompson representative. We will denote this group as $\mathcal{HT}(\mathcal{T}(G,R))$. Note that this group is the same as the full group of finite one-sided shift groupoid on the space of rooted paths in $G$. These groups have been discussed in \cite{matui2013} and are important examples of profinite groups acting on a Cantor space. 
    
    We will define the Leavitt path algebra as in \cite{abrams2015leavitt}.
    \begin{defn}
        For any graph $G$ and any ring, $\mathcal{R}$ we can define the Leavitt path algebra $L_{\mathcal{R}}(G)$ to be the (a priori non-unital) algebra over $\mathcal{R}$ generated by the set:
            \[VG\sqcup\{e,e^*\mid e\in EG\}\]
        and the relations for each $e',e\in EG$ and any $v\in VG$:
        \begin{itemize}
            \item $e^*e'=   \begin{cases}
                                0, & e\neq e'\\
                                t(e), & e=e'
                            \end{cases}$
            \item $o(e)e=et(e)=e$
            \item $t(e)e^*=e^*o(e)=e^*$
            \item If $o^{-1}(v)\neq \emptyset$, $v=\sum_{e\in o^{-1}(v)} ee^*$
        \end{itemize}
    \end{defn}

    We can expand $*$ into a linear involution s.t. for any $x,y\in L_{\mathcal{R}}(G)\  (xy)^*=y^*x^*$. By setting $v^*:=v$, $(e^*)^*:=e$, for any $x_1,\dots x_n\in VG\cup EG\cup EG^*$
        \[(x_1x_2\dots x_n)^*:=x_n^*\dots x_2^*x_1^*\]
    and expanding $*$ linearly to all of $L_{\mathcal{R}}(G)$. This makes $L_{\mathcal{R}}(G)$ into an $*$-algebra.
    
     We will reinterpret each path $p\in \mathcal{W}(G)$ as an element in $L_{\mathcal{R}}(G)$ by looking at it as the product of its edges. For any path $p=e_1\dots e_n$, we will have $p^*=e^*_n\dots e^*_1$, we will call this a \textbf{ghost path}. If the path is empty, we will interpret $\varepsilon_v$ as $\varepsilon_v=v$, for any vertex $v$. 
     
    %Note also that any time we talk about bases in this chapter, they will be bases of some cofinite inescapable subspace and thus  finite, maximal independent sets. We will thus refer to them as finite, maximal bases. 
    \section{Properties of the Leavitt path algebra}
    The following equalities hold in all Leavitt path algebras of finite graphs:
    \begin{itemize}\label{ch4-prop1}
        \item $\forall v,w\in VG,\ vw=  \begin{cases}
                                                0, & v\neq w\\
                                                v, & v=w
                                            \end{cases}$
        \item $\forall v\in VG, e\in EG,$ we have
        \[(e\notin t^{-1}(v)\implies ev=v^*e^*=0)\land (e\notin o^{-1}(v)\implies ve=e^*v^*=0)\]
        \item $\forall e,f\in EH,\ (t(e)\neq o(f))\implies(ef=f^*e^*=0)$
        \item If $VG$ is finite, $\forall x\in L_{\mathcal{R}}(G),\ x\big(\sum_{v\in VG}v\big)=\big(\sum_{v\in VG}v\big)x=x$
        \item $\forall p,q\in\mathcal{W}(G),\ q^*p= \begin{cases}
                                                        T(p),   & p=q\\ 
                                                        r,      & q\prec p\land p=qr\\
                                                        r^*,     & p\prec q\land q=pr\\
                                                        0,      & \text{ otherwise }
                                                    \end{cases}$
        \item $\forall p,q\in\mathcal{W}(G),\ T(p)\neq T(q)\implies pq^*=0$
    \end{itemize}
    We can see from these identities that the set $\{pq^*\mid p,q\in \mathcal{W}(G),\ T(p)=T(q)\}$ generates the Leavitt path algebra as a $\mathcal{R}$-module.
    
    We introduce a standing assumption that if we write an element of the Leavitt path algebra as:
        \[x=\sum_{m\in M,n\in N} k_{m,n}mn^*\]
    for any finite $M,N\subseteq\mathcal{W}(G)$ and $k_{m,n}\in \mathcal{R}$, we assume that for $m,n$ with $T(m)\neq T(n)$, $k_{m,n}=0$. This can be done, since in that case $mn^*=0$ (which can be seen from the above identities).
    In order to identify when a sum in this algebra is non-trivial, we will introduce a homomorphism $\pi$ of $L_{\mathcal{R}}(G)$ on the free $\mathcal{R}$-module $\mathcal{M}$ generated by the set of symbols $\{ X_p\mid p\in\mathcal{W}(G)\}$. This will generalise the construction from \cite{BROWNLOWE2016}. We will define the $\mathcal{M}$-endomorphisms $\Lambda_e,\Lambda^*_e,\Lambda_v$ for each $e\in EG, v\in VG$, by setting for each $p\in\mathcal{W}(G)$:
    \begin{itemize}
        \item $\Lambda_e(X_p)= \begin{cases}
                        X_{ep}, & t(e)=O(p)\\
                        0,  & \text{otherwise}
                    \end{cases}$
        \item $\Lambda^*_e(X_p)= \begin{cases}
                        X_q, & \exists q\in \mathcal{W}(G),\ p=eq\\
                        0,  & \text{otherwise}
                    \end{cases}$
        \item $\Lambda_v(X_p)=    \begin{cases}
                        X_p, & O(p)=v\\
                        0,  & \text{otherwise}
                    \end{cases}$.
    \end{itemize}
    We can see that these functions satisfy the following, for any edges $e,e'$ and any vertex $v$:
    \begin{itemize}
        \item $\Lambda^*_e\circ \Lambda_{e'}=   \begin{cases}
                                                    0, & e\neq e'\\
                                                    \Lambda_{t(e)}, & e=e'
                                                \end{cases}$
            \item $\Lambda_{o(e)}\circ \Lambda_e=\Lambda_e\circ \Lambda_{t(e)}=\Lambda_e$
            \item $\Lambda_{t(e)}\circ \Lambda^*_e=\Lambda^*_e\circ \Lambda_{o(e)}=\Lambda^*_e$
            \item If $o^{-1}(v)\neq \emptyset$, $\Lambda_v=\sum_{e\in o^{-1}(v)} \Lambda_e\circ \Lambda^*_e$.
    \end{itemize}
Thus, we can construct a homomorphism $\pi:L_{\mathcal{R}}(G)\to \End(M)$ by setting: $\pi(e):=\Lambda_e$, $\pi(e^*):=\Lambda^*_e$ and $\pi(v):=\Lambda_v$, and expanding $\pi$ to the whole algebra.  Since $\pi$ is a homomorphism, we have for each path $p$: $\pi(p)=\Lambda_p$ and $\pi(p^*)=\Lambda^*_p$, which are defined as:
        \[\forall q\in \mathcal{W}(G),\  
        \Lambda_p(X_q)=\begin{cases}
                        X_{pq}, & T(p)=O(q)\\
                        0,      & \text{otherwise}
                    \end{cases}\text{ and }\ 
        \Lambda^*_p(X_q)=\begin{cases}
                        X_{r}, & q=pr\\
                        0,      & \text{otherwise}
                    \end{cases}.\]
    This homomorphism allows us to show a slightly stronger version of \cite[Proposition 4.9.]{TOMFORDE2011471}.
    \begin{lemma}\label{ch4-lem1.5}
        The set $\{p\mid p\in \mathcal{W}(G)\}\cup\{p^*\mid p\in \mathcal{W}(G)\}$ is linearly independent in $L_{R}(G)$.
    \end{lemma}
    \begin{proof}
        To show that for $P:=\{p\mid p\in \mathcal{W}(G)\}$, $P\cup P^*$ is linearly independent, we take some finite $A\subseteq P$ and some $k_p,k_{p^*}\in \mathcal{R}$ s.t.
            \[\sum_{p\in A} k_p p+\sum_{p\in A} k_{p^*} p^*=0,\]
        and show that all the coefficients are equal to $0$.
        
        For any empty paths $\varepsilon_v$ in $A$, we can assume $k_{\varepsilon^*_v}=0$ since $\varepsilon_v=\varepsilon^*_v$.\\
        Since $\pi$ is a homomorphism, we have
            \[0=\pi(\sum_{p\in A} k_p p+\sum_{p\in A} k_{p^*} p^*)=\sum_{p\in A} k_p\pi(p)+\sum_{p\in A} k_{p^*} \pi(p^*)=\sum_{p\in A} k_p\Lambda_p+\sum_{p\in A} k_{p^*} \Lambda^*_p,\]
         and if we evaluate this at $X_{\varepsilon_{v}}$ for any $v\in VG$ we get
            \[0=\sum_{p\in A} k_p\Lambda_p(X_{\varepsilon_{v}})+\sum_{p\in A} k_{p^*} \Lambda^*_p(X_{\varepsilon_{v}})=\sum_{p\in A, T(p)=v} k_pX_p.\]
        As the module is free and $v$ arbitrary, we must have $k_p=0$ for each $p\in A$.
        
        Take some $q\in A$, s.t. for any other $p\in A$ we have $q\not\prec p$ ($q$ has maximal length) then:
            \[0=\sum_{p\in A} k_{p^*} \Lambda^*_p(X_q)=k_{q^*}X_{\varepsilon_{T(q)}}\]
        giving us $k_{q^*}=0$. By removing $q$ from $A$ and repeating this step, we inductively get for each $p\in A$: $k_{p^*}=0$. This gives us the desired linear independence.
    \end{proof}
    We can define for any $R\in VG$,  $L_{\mathcal{R}}(G,R)$ to be the $\mathcal{R}$-submodule of $L_{\mathcal{R}}(G)$ generated by 
    \[\{pq^*\mid p,q\in \mathcal{W}(G),\ T(p)=T(q),\ O(p)=O(q)=R\}.\] We can see that $L_{\mathcal{R}}(G,R)$ is closed under the $*$ operation. Additionally,  the properties satisfied by the Leavitt path algebra show that it is also closed under multiplication. So it is a $*$-subalgebra of $L_{\mathcal{R}}(G)$. Since $R=\sum_{e\in o^{-1}(R)} ee^*\in L_{\mathcal{R}}(G,R)$ and for each $p,q\in\mathcal{W}(G,R)$
        \[Rpq^*=pq^*R=pq^*,\]
    $R$ is the unit in $L_{\mathcal{R}}(G,R)$. To connect the finite maximal bases of $\mathcal{T}(G,R)$ with the elements of $L_{\mathcal{R}}(G,R)$, we will define a simple expansion of a basis as follows.
    \begin{defn}
        For any basis $B\subseteq\mathcal{T}(G,R)$ and any $p\in B$, we define the simple expansion of $B$ based on $p$, as
            \[B^p=\big(B\setminus\{p\}\big)\cup\{pe\mid e\in o^{-1}(T(p))\}.\]
    \end{defn}
    We note that a simple expansion of a basis is also a basis and for any two bases $B_1,B_2\subseteq\mathcal{T}(G,R)$ with $\mathcal{S}(B_2)\subseteq\mathcal{S}(B_1)$ we can get $B_2$ out of $B_1$ by a series of simple expansions. A special case of this is that any basis can be reached by a series of simple expansions from $\{\varepsilon_R\}$.
    \begin{lemma}\label{ch4-lem1}
        For any finite maximal basis $B$ in $\mathcal{T}(G,R)$, we have:
            \[R=\sum_{p\in B}pp^*\]
        in $L_{\mathcal{R}}(G,R)$ for each ring $\mathcal{R}$.
    \end{lemma}
    \begin{proof}
        We will inductively prove this by noting that for the basis $B=\{\varepsilon_R\}$ we have
            \[\sum_{p\in B} pp^*=\varepsilon_R\varepsilon_R^*=RR^*=R.\]
         Since any basis can be achieved via simple expansion from $\{\varepsilon_R\}$, we simply have to show that if the lemma holds for some basis $B$ it also holds for $B^q$ for any $q\in B$. This can be seen by noting that
            \[\sum_{r\in B'}rr^*=\sum_{p\in B\setminus\{q\}}pp^*+q\Big(\sum_{e\in o^{-1}(T(q))}ee^*\Big)q^*=\sum_{p\in B\setminus\{q\}}pp^*+qT(q)q^*=\sum_{p\in B}pp^*.\]
        So by induction,
            \[\sum_{r\in B}rr^*=R\]
        for any basis $B$ in $\mathcal{T}(G,R)$.
    \end{proof}
    The two above lemmas will allow us to find linearly independent sets of the rooted Leavitt path algebra that together span the whole algebra.
    \begin{lemma}\label{ch4-lem2}
        For any basis $B$, the sets $\{bp^*\mid b\in B,\ p\in \mathcal{W}(G,R),\ T(b)=T(p)\}$ and $\{pb^*\mid b\in B,\ p\in \mathcal{W}(G,R),\ T(b)=T(p)\}$ are linearly independent.
    \end{lemma}
    \begin{proof}
        Note first that since the paths inside the basis are not prefixes of each other we have for any $b,b'\in B$
            \[(b')^*b=  \begin{cases}
                            0, & b\neq b',\\
                            T(b), & b=b'
                        \end{cases}.\]
        To show the linear independence, we look at some finite set $M\subseteq \mathcal{T}(G,R)$ and some basis $B$ of an inescapable cofinite subspace of $\mathcal{T}(G,R)$  and some $k_{b,p}\in \mathcal{R}$ for each $(b,p)\in B\times M$ with $T(b)=T(p)$ s.t.
            \[\sum_{\substack{b\in B, p\in M,\\ T(b)=T(p)}} k_{b,p} bp^*=0.\]
        Since the paths in $B$ are independent of each other, we can take some $b_0\in B$ and multiply by $b^*_0$ to get
            \[0=b^*_0\Big(\sum_{\substack{b\in B,p\in M,\\T(b)=T(p)}} k_{b,p} bp^*\Big)=\sum_{p\in M,T(b_0)=T(p)} k_{b_0,p} p^*. \]
        So by \lemref{ch4-lem1.5} we must have $k_{b_0,p}=0$ for each $p\in M, b_0\in B$ with $T(b)=T(p)$. This gives us the independence of the first set.

        We note that the second set is the conjugate of the first, so its linear independence follows from the linear independence of the first.
        % For the second claim, we note that since $\{pq^*\mid p,q\in \mathcal{W}(G)\land T(p)=T(q)\land O(p)=O(q)=R\}$ spans $L_{\mathcal{R}}(G,R)$ there are some finite sets $L,N\subseteq \mathcal{W}(G,R)$ and $k_{p,q}\in\mathcal{R}$ for each $p\in N,q\in L$ s.t.
        %     \begin{equation}\label{ch4-eq1}
        %         x=\sum_{p\in L,q\in N} k_{p,q} pq^*.
        %     \end{equation}
        % Take a basis $B$ of $\mathcal{T}(G,R)$ s.t. $\forall p\in L,\exists b\in B,\ p\preceq b$. For any $p\in L$ if we look at the set $B_p=\{r\in \mathcal{W}(G)\mid \exists b\in B,\ b=pr\}$ it is a basis in $\mathcal{T}(G_{T(p)},T(p))$ and thus by \lemref{ch4-lem1} we have for each $q\in N$:
        %     \[ pq^*=pT(p)q^*=p(\sum_{r\in B_p}rr^*)q^*=\sum_{r\in B_p}pr(qr)^* \]
        % and since for each $r\in B_p$: $pr\in B$, when we plug this into the sum from \ref{ch4-eq1} we get:
        %     \[x=\sum_{b\in B,m\in M} l_{b,m} bm^*,\]
        % where $B$ is a basis of $\mathcal{T}(G,R)$, $M\subseteq\mathcal{W}(G,R)$ is a finite set and $l_{b,m}$ are some elements of $ \mathcal{R}$.
    \end{proof}
    We can show that the union of these linearly independent sets from the above lemma in fact span the whole Leavitt path algebra.
    \begin{lemma}\label{ch4-lem2.1}
        Let
            \[x=\sum_{m\in M, n\in N}k_{m,n}mn^*\in L_{\mathcal{R}}(G,R)\]
        for some finite $M,N\subseteq \mathcal{T}(G,R)$ and $k_{m,n}\in \mathcal{R}$. Then for any finite maximal basis $B\subseteq\mathcal{T}(G,R)$:
        \begin{itemize}
            \item if for each $m\in M$ there is some $b\in B$ s.t. $m\preceq b$, then
                \[x\in\text{span}(\{bp^*\mid b\in B,p\in\mathcal{W}(G,R)\}).\]
            \item  if for each $n\in N$ there is some $b\in B$ s.t. $n\preceq b$, then
                \[x\in\text{span}(\{pb^*\mid b\in B,p\in\mathcal{W}(G,R)\}).\]
        \end{itemize}
    \end{lemma}
    \begin{proof}
        For the first point we can define for any $m\in M$, $B_m:=\{r\in \mathcal{W}(G)\mid mr\in B\}$. Since $B$ is a basis that does not contain a prefix of $m$ (since it contains $b$ s.t. $m\preceq b$), $B_m$ must be a basis in $\mathcal{T}(G_{T(m)},T(m))$. Thus, if we apply \lemref{ch4-lem1}, we get for any $n\in N$
            \[mn^*=mT(m)n^*=m\Big(\sum_{r\in B_m}rr^*\Big)n^*=\sum_{r\in B_m}(mr)(nr)^*.\]
        So by definition of $B_m$, we have $mn^*\in \text{span}(\{bp^*\mid b\in B,p\in\mathcal{W}(G,R)\}) $. Thus, we must have 
            \[x\in\text{span}(\{bp^*\mid b\in B,p\in\mathcal{W}(G,R)\}).\]
        The second point follows analogously.
    \end{proof}
    Unfortunately, we cannot strengthen this lemma, as not every element in the Leavitt path algebra is in the submodule generated by $\{pq^*\mid p\in B,\ q\in B',\ T(p)=T(q)\}$ for some finite maximal bases $B,B'$. An example of an element not in any of these submodules would be $p+q$ for any paths $p\precneq q$. We will call the elements that are in the span of $\{pq^*\mid p\in B,\ q\in B',\ T(p)=T(q)\}$, for some bases, \textbf{matrix-like}. Their resemblance to matrices will become more clear as we go along.\\
    While the set of matrix-like elements is not closed under addition (as we can add two paths that are prefixes to each other), we can see that it is closed under multiplication.

    Now we will show what happens when linear combinations in the spanning set $\{pq^*\mid p,q\in\mathcal{W}(G,R)\}$ are equal to zero.
    \begin{lemma}\label{ch4-lem2.5}
        For any finite sets $M,N\subseteq \mathcal{W}(G,R)$ and any coefficients $k_{m,n}$ for $m\in M,n\in N$, if we have
            \[\sum_{m\in M,n\in N} k_{m,n} mn^*=0. \]
        then for any $p\in \mathcal{W}(G,R)$, s.t. $\forall m\in M\cup N, p\not\prec m$, we have
            \[\sum_{m,n\preceq p}k_{m,n}=0.\]
        Additionally if we have
            \[\sum_{m\in M,n\in N} l_{m,n} mn^*= \sum_{m\in M,n\in N} j_{m,n} mn^*, \]
        then
            \[\sum_{m,n\preceq p}l_{m,n}=\sum_{m,n\preceq p} j_{m,n}.\]
    \end{lemma}
    \begin{proof}
        Take $p\in \mathcal{W}(G,R)$ such that $\forall m\in M\cup N, p\not\prec m$ then we have
            \[0=p^*(\sum_{m\in M,n\in N} k_{m,n} mn^*)p=\sum_{\substack{m\preceq p, n\leq p\\ p=mr=ns}}k_{m,n}r^*s.\]
        Note that for each of the $(r,s)$ pairs one of the paths is a prefix of the other since they are both suffixes of $p$. So $0\neq r^*s\in\{p,p^*\mid p\in\mathcal{W}(G,R)\}$,  which is linearly independent. Thus, we must have
            \[\sum_{m,n\preceq p,}k_{m,n}=0.\]
        If we have
            \[\sum_{m\in M,n\in N} l_{m,n} mn^*= \sum_{m\in M,n\in N} j_{m,n} mn^*,\]
        then
            \[\sum_{m\in M,n\in N} (l_{m,n}-j_{m,n}) mn^*=0,\]
        and thus, by applying the previous result
            \[\sum_{m,n\preceq p}l_{m,n}=\sum_{m,n\preceq p}j_{m,n}.\]
    \end{proof}
    \section[Embedding the Higman--Thompson group]{Embedding the Higman--Thompson group into the Leavitt path algebra}
    
    We can define the group of unitary elements in $L_{\mathcal{R}}(G,R)$ by
        \[U(L_{\mathcal{R}}(G,R)):=\{x\in L_{\mathcal{R}}(G,R)\mid xx^*=x^*x=1_{L_{\mathcal{R}}(G,R)}=R\},\]
    this is a multiplicative group with $*$ being the inverse.
    We will show in this chapter that we can embed the Higman--Thompson group into $U(L_{\mathcal{R}}(G,R))$, we will then construct a right inverse of this embedding that will induce an isomorphism between the Higman--Thompson groups whenever the Leavitt path algebras are isomorphic.
    
    Note, that we will be working mostly with $\mathds{Z}$, since this allows us to very easily work with the orthogonal matrices.
    To prove this theorem we will need several lemmas that tie Higman--Thompson automorphisms to the Leavitt path algebras.
    \begin{lemma}\label{ch4-lem3}
        For any Higman--Thompson representative $\phi:\mathcal{S}(B)\to\mathcal{S}(B')$, we have
            \[\sum_{p\in B}\phi(p)p^*\in U(L_{\mathcal{R}}(G,R))).\]
    \end{lemma}
    In the following proof we will use $\delta_{p,q}$ to be the Kronecker delta, i.e. set it equal to $1$ whenever $p=q$ and $0$ otherwise. Also recall that $\mathcal{W}^+(G)$ denotes the space of non-empty paths.
    \begin{proof}
        Note that since $\phi$ is a Higman--Thompson representative, we must have $T(\phi(p))=T(p)$ for any path $p\in \mathcal{S}(B)$. Using this we can calculate   
        \begin{multline*}
            \Big(\sum_{p\in B}\phi(p)p^*\Big)\Big(\sum_{p\in B}\phi(p)p^*\Big)^*=\Big(\sum_{p\in B}\phi(p)p^*\Big)\Big(\sum_{p\in B}p\phi(p)^*\Big)=\\
            \sum_{p\in B}\phi(p)p^*p\phi(p)^*=\sum_{p\in B}\phi(p)T(p)\phi(p)^*=\sum_{p\in B}\phi(p)T\big(\phi(p)\big)\phi(p)^*=\sum_{p\in B}\phi(p)\phi(p)^*=R,
        \end{multline*}
        where the third equation follows since any $p\neq q\in B$ are independent, so $p^*q=0$ and  the last equation follows from \lemref{ch4-lem1} and the fact that $\phi$ is bijective. The other equation 
            \[\Big(\sum_{p\in B}\phi(p)p^*\Big)^*\Big(\sum_{p\in B}\phi(p)p^*\Big)=1_{L_{\mathcal{R}}(G,R)},\]
        follows by using the above calculation for $\phi^{-1}$.
    \end{proof}
    We can associate any Higman--Thompson representative $\phi$ with the unitary element $\sum_{p\in B}\phi(p)p^*$. To better understand this association in the case where $\mathcal{R}=\mathds{Z}$ we will have to describe the elements of $U(L_{\mathds{Z}}(G,R))$.
    %We will need a technical lemma for this:
    % \begin{lemma}\label{ch4-lem4}
    %     For any cofinite inescapable bases $B,B'$ of $\mathcal{T}(G,R)$ the set $\{b(b')^*\mid b\in B,b'\in B',T(b)=T(b')\}$ is linearly independent.
    % \end{lemma}
    % \begin{proof}
    %     Take some linear combination in the set that is equal to $0$, i.e.:
    %         \[x:=\sum_{b\in B,b'\in B'}k_{b,b'}b(b')^*=0\]
    %     where $k_{b,b'}\in \mathcal{R}$ and $k_{b,b'}=0$ if $T(b)\neq T(b')$. For any $b_0\in B$, $b'_0\in B'$ we have for all $b\in B$ and  $b'\in B'$:
    %         \[b_0^*b(b')^*b_0= \begin{cases}
    %                                 T(b_0), & b=b'=b_0\\
    %                                 0,      & \text{ otherwise}
    %                             \end{cases}\]
    %     so we must have:
    %         \[0=b_0^*xb'_0=k_{b_0,b'_0}T(b_0)\]
    %     so we have $k_{b_0,b'_0}=0$.
    %     This shows linear independence.
    % \end{proof}
    The following will show that all unitary elements are matrix-like.
    \begin{lemma}\label{ch4-lem5}
        For any rooted graph $(G,R)$, we have:
            \begin{multline*}
                U(L_{\mathds{Z}}(G,R))=
                \{ \sum_{b\in B,b'\in B'}k_{b,b'}b(b')^*\mid (k_{b,b'})_{b,b'}\in \mathcal{O}(\mathds{Z}),\ (T(b)\neq T(b'))\implies (k_{b,b'}=0)\}
            \end{multline*}
        where $\mathcal{O}(\mathds{Z})$ denotes the set of orthogonal matrices over $\mathds{Z}$ (in any dimension).
    \end{lemma}
    Note that the bulk of this proof will work for arbitrary rings. Especially, we show the ``$\supseteq$" inclusion for arbitrary rings.
    \begin{proof}
        We will first show that for any two finite maximal bases $B,B'$ of $\mathcal{T}(G,R)$ and any orthogonal matrix $K=(k_{b,b'})_{b\in B,b'\in B'}$, the element
            \[x=\sum_{b\in B,b'\in B'}k_{b,b'}b(b')^*\]
        is unitary. For this, we calculate
            \[xx^*=\Big(\sum_{b\in B,b'\in B'}k_{b,b'}b(b')^*\Big)\Big(\sum_{c\in B,c'\in B'}k_{c,c'}c'c^*\Big)=\sum_{b,c\in B}\Big(\sum_{b'\in B'}k_{b,b'}k_{c,b'}\Big)bc^*=\sum_{b\in B}bb^*=1,\]
        which follows from $K$ being orthogonal and $B'$ being a basis. Showing that $x^*x=1$ follows analogously from $B$ being a basis.
        
        For the converse inclusion, we have to show that any unitary element $x$ is of that form. We take $B$ to be a finite basis and $M\subseteq \mathcal{W}(G,R)$ to be a finite set s.t. by \lemref{ch4-lem2.1} we can write
            \[x=\sum_{b\in B,m\in M}k_{b,m}bm^*\]
        with $M$ being chosen s.t. for each $m\in M$ we have some $b\in B$ s.t. $k_{b,m}\neq 0$. As $x$ is unitary, we have
            \[1=x^*x=\sum_{m,n\in M}\Big(\sum_{b\in B}k_{b,m} k_{b,n}\Big)mn^*.\]
        Take $C$ to be a finite maximal basis s.t. $\forall m\in M,\exists c\in C\ m\preceq c $, by \lemref{ch4-lem1} we have
            \[\sum_{c\in C} cc^*=\sum_{m,n\in M}\Big(\sum_{b\in B}k_{b,m} k_{b,n}\Big)mn^*.\]
        So by \lemref{ch4-lem2.5} we have for any $c\in C$
            \[\sum_{m,n\preceq c}\Big(\sum_{b\in B}k_{b,m} k_{b,n}\Big)=1.\]
        If we define for each $m\in M$ the $|B|$-dimensional vector $\vec{k}_m=(k_{b,m})_{b\in B}$, we can write this as
            \[\Big(\sum_{m\preceq c} \vec{k}_m\Big)\cdot\Big(\sum_{n\preceq c} \vec{k}_n\Big)=\sum_{m\preceq c,n\preceq c}\vec{k}_m\cdot \vec{k}_n= 1.\]
        So the vectors $(\sum_{m\preceq c} \vec{k}_m)$ are $|B|$-dimensional vectors of length $1$.
        
        When we look at the second equality $xx^*=1$ we get for any $b_1,b_2\in B$
            \[\delta_{b_1,b_2}T(b_1)=b^*_1b_2=b^*_1xx^*b_2=\sum_{m,n\in M}k_{b_1,m}k_{b_2,n}m^*n.\]
        % and if we again consider that $m^*n\in \{p,p^*\mid p\in \mathcal{W}(G,R)\}$ which is a set of linearly independent elements we have:
        %     \[\sum_{m,n\in M}k_{b_1,m}k_{b_2,n}=\delta_{b_1,b_2}\]
        % And if we define $\vec{k}_{b}=(k_{b,m})_{m\in M}$. We note that if $\vec{k}_{b}=0$ we would have $b^*x=0$, but then: $T(b)=b^*xx^*b=0$ which is a contradiction.
        By sorting the summands in the above equality we get
            \begin{multline*}
                \delta_{b_1,b_2}T(b_1)=\sum_{v\in VG}\Big(\sum_{\substack{m\in M\\ T(m)=v}}k_{b_1,m}k_{b_2,m}\Big)v+\sum_{r\in\mathcal{W}^+(G)}\Big(\sum_{\substack{m,n\in M\\ m=nr}} k_{b_1,m}k_{b_2,n}\Big) r\\+\sum_{r\in\mathcal{W}^+(G)}\Big(\sum_{\substack{m,n\in M\\ mr=n}} k_{b_1,m}k_{b_2,n}\Big) r^*
            \end{multline*}
        So if we define $\vec{k}_{b,r}:=(k_{b,m})_{m,mr\in M}$ and $\vec{k}_{b,r^*}:=(k_{b,mr})_{m,mr\in M}$ for each $b\in B,\ r\in\mathcal{W}(G)$, while ordering them s.t. $k_{b,m}$ has the same position in $\vec{k}_{b,r}$ as $k_{b,mr}$ in $\vec{k}_{b,r^*}$, we must have
            \[\vec{k}_{b_1,r}\cdot \vec{k}_{b_2,r^*}=\delta_{b_1,b_2}\delta_{\varepsilon_{T(b_1)},r}.\]
        Note that thus for any $b\in B$ and $v\in VG$
             \[\vec{k}_{b,\varepsilon_{v}}\cdot \vec{k}_{b,\varepsilon^*_{v}}=\delta_{T(b),v},\]
         which means, since $\vec{k}_{b,\varepsilon_{v}}=\vec{k}_{b,\varepsilon^*_{v}}$, those are vectors of length $1$. Additionally, since by assumption, these vectors are in $\mathds{Z}^{|M|}$, there must be a unique $m_b\in M$ with $T(b)=T(m_b)$ s.t. $k_{b,m_b}\neq 0\iff m=m_b$. Conversly any $m\neq m_b$, $k_{m,b}=0$, however since any for any $m$ some $k_{m,b}$ must be non-zero, it means that any $m$ is an $m_b$ for some $b\in B$.
        Now take some $c\in C$ and note that this implies that
            \[\forall b\in B,\ \Big(\sum_{m\preceq c} \vec{k}_{m}\Big)_b=   \begin{cases}
                                                    k_{b,m_b}, & m_b\leq c\\
                                                    0,  &\text{otherwise}
                                                \end{cases},\]
        and since this has norm $1$ we must have a unique $b$ s.t. $m_b\preceq c$. 
        
        %Additionally, by assumption, we must have for each $m\in M$ some $b\in B$ s.t. $k_{b,m}\neq 0$ and thus $m=m_b$. 
        Thus, for any $c\in C$, we have a unique $m\in M$ s.t. $m\preceq c$. Thus, $M$ must be a basis, we may assume that $M=C$. So we can write
            \[x=\sum_{b\in B,c\in C} k_{b,c}bc^*,\]
        making $x$ matrix-like.
        As $x$ is unitary we have
            \[\sum_{b\in B}bb^*=1=xx^*=\sum_{b_1,b_2\in B}\Big(\sum_{c\in C}k_{b_1,c}k_{b_2,c}\Big)b_1b^*_2.\]
        So by linear independence if we define $K=(k_{b,c})_{b\in B,c\in C}$ we must have $KK^T=I_{|B|}$ and similarly if we consider $1=x^*x$ we have $K^TK=I$. Thus $K$ is orthogonal.\\

    \end{proof}
    This allows us to embed the Higman--Thompson group of $\mathcal{T}(G,S)$ into the unitary group $U(L_{\mathcal{R}}(G,S))$:
    \begin{lemma}\label{ch4-lem6}
        For any rooted graph $(G,R)$ and any ring $\mathcal{R}$ there is an injective homomorphism:
            \[i:\mathcal{HT}(\mathcal{T}(G,S))\hookrightarrow U(L_{\mathcal{R}}(G,R))\]
    \end{lemma}
    \begin{proof}
        For any Higman--Thompson representative $\phi:\mathcal{S}(B)\to \mathcal{S}(B')$ define the element of the Leavitt path algebra
            \[x_{\phi}:=\sum_{b\in B}\phi(b)b^*.\]
        We will define
            \[i([\phi]):=x_{\phi}.\]
        To see that this is well-defined we just have to show that any simple expansion of the basis defines the same $x_{\phi}$. For this take some $p\in B$ and set $\tilde{\phi}=\phi|_{\mathcal{S}(B^p)}$, we get:
        \begin{multline*}
            x_{\phi}=\sum_{b\in B}\phi(b)b^*=\sum_{\substack{b\in B,\\ b\neq b_0}}\phi(b)b^* + \phi(b_0)T(b_0)b_0^*=\\
            \sum_{\substack{b\in B,\\ b\neq b_0}}\phi(b)b^* + \sum_{e\in o^{-1}(T(b_0))}\phi(b_0)ee^*b_0^*=\sum_{b\in \tilde{B}}\phi(b)b^*=x_{\tilde{\phi}}
        \end{multline*}
        Thus $i$ is well-defined. To show that it is a homomorphism, we take two compatible Higman--Thompson representatives $\phi:\mathcal{S}(B_1)\to \mathcal{S}(B_2)$ and $\psi:\mathcal{S}(B_1)\to \mathcal{S}(B_3)$ and note that since $B_2$ is a basis and $T(\phi(b))=T(b)$ for any $b\in B_1$
            \begin{multline*}
                x_{\phi}(x_{\psi})^{-1}=x_{\phi}(x_{\psi})^*=\Big(\sum_{b\in B_1}\phi(b)b^*\Big)\Big(\sum_{b\in B_1}b\psi(b)^*\Big)=\\
                \sum_{\substack{b'\in B_2,b''\in B_3\\ \phi^{-1}(b')=\psi^{-1}(b)}}b'(b'')^*=\sum_{c\in B_3}\phi\big(\psi^{-1}(c)\big)(c)^*=x_{\phi\circ\psi^{-1}}=i([\phi][\psi]^{-1}).
            \end{multline*}
        Finally to see that it is injective we note that the unit in $U(L_{\mathcal{R}}(G,R))$ is also the unit in $L_{\mathcal{R}}(G,R)$ so for any $\phi:\mathcal{S}(B)\to \mathcal{S}(B')$ s.t. $x_{\phi}=1_{L_{\mathcal{R}}(G,R)}$ we must have for any $b_0\in B$
            \[(b_0)^*=(b_0)^*x_{\phi}=\sum_{b\in B}(b_0)^*b\phi(b)^*=T(b_0)\phi(b_0)^*=\phi(b_0)^*.\]
        By \lemref{ch4-lem2} we must therefore have $b_0=\phi(b_0)$ (as paths) for each $b_0\in B$. Thus $[\phi]$ is the identity.
    \end{proof}
    Define, for finite maximal bases $B,B'$ and matrices $M$ over $\mathds{Z}$, that are indexed by $B\times B'$
        \[x_{B,B',M}=\sum_{b\in B,b'\in B'}M_{b,b'}b(b')^*.\]
    While this definition depends on the ordering of $B$ and $B'$, one can endow each finite maximal basis with a canonical order to make it well-defined. Crucially, if $B=B'$ we assume that they are ordered the same way. Additionally, to make this unique, we assume that $T(b)\neq T(b')\implies M_{b.b'}=0$.
    From \lemref{ch4-lem5} we can see that \[U_{L_{\mathds{Z}}(G,S)}=\{x_{B,B',M}\mid M\text{ is an orthogonal matrix}\}\] and
    \[i\Big(\mathcal{HT}\big(\mathcal{T}(G,S)\big)\Big)=\{x_{B,B',M}\mid M\text{ is a permutation matrix}\}.\] Since the ring we are working with is $\mathds{Z}$, we note that all the orthogonal matrices are permutation matrices multiplied by diagonal matrices (with entries $\pm 1$ on the diagonal). This allows us to characterise the Higman--Thompson group as a factor of $U\big(L_{\mathds{Z}}(G,S)\big)$. We will factor the subgroup of diagonal unitary elements
        \[DU\big(L_{\mathds{Z}}(G,R)\big):=\{\sum_{b\in B} \kappa_b bb^*\mid B\text{ is a finite maximal basis}, \kappa_{b}\in\{1,-1\}\}.\]
    \begin{lemma}\label{ch4-lem7}
        For any rooted graph $(G,R)$ we have
            \[\mathcal{HT}\big(\mathcal{T}(G,R)\big)\cong U\big(L_{\mathds{Z}}(G,R)\big)/{DU\big(L_{\mathds{Z}}(G,R)\big)}.\]
    \end{lemma}
    \begin{proof}
        Since any orthogonal matrix in $\mathds{Z}$ can be obtained from a permutation matrix by changing some $1$'s to $-1$'s we can write any $x\in U(L_{\mathds{Z}}(G,R))$ as
            \[x=\sum_{b\in B}\kappa_b \phi(b)b^*,\]
        where $B$ is a finite maximal basis, $\phi:\mathcal{S}(B)\to \mathcal{S}(B')$ is a Higman--Thompson representative and $\kappa_b=\pm 1$. We can thus define the map $\Theta: U\big(L_{\mathds{Z}}(G,R)\big)\to \mathcal{HT}(\mathcal{T}(G,R))$ s.t. $\Theta(x)=[\phi]$. To see that it is well-defined, we take some Higman--Thompson representatives $\phi:\mathcal{S}(B_1)\to \mathcal{S}(B_2)$ and $\psi:\mathcal{S}(B_3)\to \mathcal{S}(B_4)$ and some $\kappa_b,\kappa'_{b'}\in \{1,-1\}$ for each $b\in B_1,b'\in B_3$ s.t.
            \[x:=\sum_{b\in B_1}\kappa_b \phi(b)b^*=\sum_{b'\in B_3}\kappa'_{b'} \psi(b')(b')^*=:y.\]
        Since expanding bases does not change these sums, we may assume that $B_2=B_4$. Additionally, since $x,y$ are unitary, we have
            \[\sum_{b\in B_1} bb^*=R=1_{L_{\mathds{Z}}(G,R)}=yx^*=\sum_{b\in B_1}\kappa_b\kappa'_{b}b\psi^{-1}(\phi(b))^*,\]
        so for any $b\in B_1$
            \[b^*=(\psi^{-1}(\phi(b))^*.\]
        Thus $\psi^{-1}\big(\phi(b)\big)=b$, as paths (using linear independence of starred paths), showing $\psi=\phi$.\\
        Similarly to the proof of \lemref{ch4-lem6} we can show that $\Theta$ is a homomorphism.
        
        $\Theta$ is clearly surjective, since for any Higman-Thompson representative $\phi$ we have:
            \[[\phi]=\Theta\Big(\sum_{b\in B}\phi(b)b^*\Big)\]
        So it suffices to show that the kernel of $\Theta$ consists of the diagonal elements. This follows from
            \[\Theta\Big(\sum_{b\in B}\kappa_b \phi(b)b^*\Big)=\text{id}\iff \phi=\text{id} \iff \sum_{b\in B}\kappa_b \phi(b)b^*\in DU\big(L_{\mathds{Z}}(G,R)\big).\]
        This gives us the desired result.
    \end{proof}
    We also note that the surjection from the above proof give us a short exact sequence
        \[0\to DU\big(L_{\mathds{Z}}(G,R)\big)\to U\big(L_{\mathds{Z}}(G,R)\big)\to\mathcal{HT}\big(\mathcal{T}(G,R)\big)\to 0.\]
    Additionally since the second map has a right inverse as seen in \lemref{ch4-lem6}, so the unitaries can be written as a semi-direct product 
        \[U\big(L_{\mathds{Z}}(G,R)\big)\cong DU\big(L_{\mathds{Z}}(G,R)\big)\rtimes \mathcal{HT}\big(\mathcal{T}(G,R)\big).\]
    To show that an isomorphism between two rooted Leavitt path algebras of graphs induces an isomorphism between the corresponding Higman--Thompson groups, we will have to use an alternative characterisation the diagonal unitary.
    For this we will define the set of symmetric elements in $L_{\mathcal{R}}(G,R)$
    \[
        S\big(L_{\mathcal{R}}(G,R)\big):=\{ x\in L_{\mathcal{R}}(G,R) \mid x^*=x \}.
        % D(L_{\mathds{Z}}(G,R))=&\{ x\in L_{\mathds{Z}}(G,R) \mid \forall y,z\in L_{\mathds{Z}}(G,R),\ (x=z+y)\implies\\
        % &\exists y',z'\in S(L_{\mathds{Z}}(G,R)),\ (z-z')=(y-y')^*=-(y-y')\}
    \]
    While we cannot characterise the symmetric elements in the same way we have for unitary ones, we can however establish two lemmas about them:
    \begin{lemma}\label{ch4-lem8}
        For any matrix-like element $x\in S\big(L_{\mathcal{R}}(G,R)\big)$ s.t. 
            \[x=\sum_{b\in B,c\in C} k_{b,c}bc^*\]
        for some finite maximal bases $B,C$, we have $B=C$ and for each $b,c\in B$
            \[k_{b,c}=k_{c,b}.\]
    \end{lemma}
    \begin{proof}
        We take some $b\in B$ and $c\in C$ with $T(b)=T(c)$ giving us
            \[b^*xc=k_{b,c}T(b).\]
        However if we take $b_c\in B$ to be the unique element that is comparable with $c$ and vice versa $c_b\in C$ to be the unique element comparable with $b$, then we have
            \[k_{b,c}T(b)=b^*xc=b^*x^*c=k_{b_c,c_b} (b^*c_b)(b_c^*c).\]
        This is only the case if $b=c_br$, $c=b_cr$ for some path $r\in\mathcal{W}(G)$ with $T(r)=T(b)$ and $k_{b,c}=k_{b_c,c_b}$. This however means that each $b\in B$ there exists a $C\ni c_b\preceq b$ and thus $\mathcal{S}(B)\subseteq\mathcal{S}(C)$. Conversely, we have for each $c\in C$ there exists a $B\ni b_c\preceq c$ and thus $\mathcal{S}(C)\subseteq\mathcal{S}(B)$. This means that $\mathcal{S}(B)=\mathcal{S}(C)$ and thus $B=C$.
        Thus, we must have $b_c=b$ and $c_b=c$ and so $k_{b,c}=k_{c,b}$, for each $b,c\in B$.\\
    \end{proof}
    For general symmetric elements, we can only say that:
    \begin{lemma}\label{ch4-lem9}
        Let $x$ be an element in $S(L_{\mathcal{R}}(G,R))$ s.t.
            \[x=\sum_{b\in B,m\in M}k_{b,m}bm^*,\]
        where $B$ is a finite, maximal basis and $M$ is a set of paths s.t. $B\subseteq M$. Then we have:
            \[\forall b,c\in B,\ k_{b,c}=k_{c,b}\]
    \end{lemma}
    \begin{proof}
        Taking some $b,c\in B$, with $T(b)=T(c)$ gives us
            \[b^*xc=\sum_{m\in M,m\prec c}k_{b,m}r_{m,c}+k_{b,c}T(c)+ \sum_{m\in M, c\prec m}k_{b,m}r_{m,c}^*,\]
        where $r_{m,c}$ is a path s.t. $c=mr_{m,c}$, or $m=cr_{m,c}$ for any $m$ that is comparable to $c$. Using symmetry of $x$ we also have:
            \[b^*xc=b^*x^*c=\sum_{m\in M, m\prec b} k_{c,m} r_{m,b} + k_{c,b}T(c)+\sum_{m\in M, b\prec m} k_{c,m} r^*_{m,b} \]
        where $r_{m,b}$ is as before.\\
        And since the set of paths and starred paths is independent due to \lemref{ch4-lem1.5}, we can compare the coefficients next to $T(c)$ and get
            \[k_{b,c}=k_{c,b}.\]
        For $b,c$ with $T(b)\neq T(c)$, we have by standing assumption $k_{b,c}=0=k_{c,b}$.
    \end{proof}
    Note additionally that each diagonal element is symmetric (as symmetric elements are closed under addition and multiplication by scalars). Using this, we can describe diagonal unitary elements as follows,
        \begin{multline*}
            \widetilde{DU}\big(L_{\mathds{Z}}(G,R)\big):=\{x\in U\big(L_{\mathds{Z}}(G,S)\big)\cap S\big(L_{\mathds{Z}}(G,R)\big)\mid \forall y\in L_{\mathds{Z}}(G,R),\\
            \exists y_+,y_-\in L_{\mathds{Z}}(G,R),\ y=y_+ + y_-,\ xy_+=y_+\ xy_-=-y_-\}.
        \end{multline*}
    We will show that this is in fact the same diagonal unitary group as defined before. 
    \begin{lemma}\label{ch4-lem10}
    For any rooted graph $(G,R)$ and any $x\in U\big(L_{\mathds{Z}}(G,R)\big)$, we have $x\in \widetilde{DU}\big(L_{\mathds{Z}}(G,R)\big)$ if and only if there exist some finite, maximal basis $B\subseteq\mathcal{T}(G,S)$ and some elements $\kappa_b\in\{\pm 1\}$ such that
        \[x=\sum_{b\in B}\kappa_b bb^*.\]
    \end{lemma}
    \begin{proof}
        To show that any element of this form is in $\widetilde{DU}\big(L_{\mathds{Z}}(G,R)\big)$, we note that for any $y\in L_{\mathds{Z}}(G,R)$ we can write it as
            \[y=\sum_{b\in B', m\in M} k_{b,m} bm^*.\]
        By expanding $B'$ and $M$ we may assume that for any $p\in B'\cup M$ we have some $b\in B$ s.t. $b\preceq p$. Since this prefix is always unique we will denote it by $b_p$ So when we multiply the equality by $x$ we get
        \begin{align*}
            xy=\sum_{p\in B',m\in M}\kappa_{b_p}k_{p,m} pm^*.
        \end{align*}
        So if $k_{p,m}=0$ whenever $\kappa_{b_p}=-1$ we have $xy=y$ and if $k_{p,m}=0$ whenever $\kappa_{b_p}=1$ we have $xy=-y$. So if we define:
        \begin{align*}
            y_+:&=\sum_{\substack{p\in B',m\in M\\ \kappa_{b_p}=1}} k_{p,m} pm^*\text{ and }\\
            y_-:&=\sum_{\substack{p\in B',m\in M\\ \kappa_{b_p}=-1}} k_{p,m} pm^*,
        \end{align*}
        we have $y=y_++y_-$, $xy_+=y_+x$ and $xy_-=-y_-x$.

        For the converse, we take some $x\in \widetilde{DU}\big(L_{\mathds{Z}}(G,R)\big)$. Since $x\in U\big(L_{\mathds{Z}}(G,S)\big)$ and is symmetric, we have (by \lemref{ch4-lem5}, \lemref{ch4-lem8} and as we are working over the integers) some finite, maximal basis $B$, a Higman--Thompson automorphism representative $\phi:\mathcal{S}(B)\to\mathcal{S}(B)$ and some $\kappa_b\in\{\pm 1\}$ for each $b\in B$  s.t.
            \[x=\sum_{b\in B} \kappa_b b\phi(b)^*.\]
        Take some element, $y\in L_{\mathds{Z}}(G,R)$ we can write it as
            \[y=\sum_{c\in C, m\in M} l_{c,m} cm^*,\]
        with $C$ being a finite maximal basis and $M$ a set s.t. $C,M\subseteq \mathcal{S}(B)$, $\phi^{-1}(C)=C$. We can achieve that by first expanding $C$ and $M$ until both $C$ and $M$ are subsets of $\mathcal{S}(B)$ and each element of $C$ has the same distance from its unique prefix in $B$ (which gives us $\phi^{-1}(C)=C$). As before, we can fix for each $p\in C\cup M$ the unique prefixes $b_p\in B$ and $b'_p\in C$. This allows us to get the following when multiplying $y$ by $x$.
        \begin{align*}
            xy=\sum_{c\in C, m\in M} \kappa_{\phi^{-1}(b_c)} l_{c,m} \phi^{-1}(c)m^*
        \end{align*}
        so we have $xy=y$ if and only if for each $c\in C,m\in M$, $\kappa_{\phi^{-1}(b_c)} l_{c,m}=l_{\phi^{-1}(c),m}$ and $xy=-y$ if and only if for each $c\in C,m\in M$, $\kappa_{\phi^{-1}(b_c)} l_{c,m}=-l_{\phi^{-1}(c),m}$. Since as assumed $x\in \widetilde{DU}(L_{\mathds{Z}}(G,R))$, we have for each $y$ two Leavitt path algebra elements $y_-,y_+$, with: $y=y_++y_-$, $xy_+=y_+$ and $xy_-=-y_-$.
        
        If we fix $y=b_0b_0^*$ for some $b_0\in B$, we can write $y_+,y_-$ as
        \begin{align*}
            y_+=&\sum_{c\in C, m\in M} l^+_{c,m} cm^*\text{ and }\\
            y_-=&\sum_{c\in C, m\in M} l^-_{c,m} cm^*,
        \end{align*}
        with $C\subseteq M\subseteq\mathcal{S}(B)$ and $\phi^{-1}(C)=C$. We can see that
        \[\sum_{c\in C, m\in M} (l^+_{c,m}+l^-_{c,m}) cm^*=y_++y_-=y=b_0b_0^*=\sum_{c\in C,  b_0\preceq c} cc^*.\]
        So for any $c\in C$ with $b_0\preceq c$ (note that such a $c$ must exist since $C$ is a basis) 
        \[1=l^+_{c,c}+l^-_{c,c}\]
        and for $b_0\not\preceq c$
        \[\forall m\in M,\ 0=l^+_{c,m}+l^-_{c,m}\]
        If $\phi(b_0)\neq b_0$ and $b_0\preceq c$ we also have $b_0 \not\preceq \phi^{-1}(c)$ and thus
        \[0=l^+_{\phi^{-1}(c),c}+l^-_{\phi^{-1}(c),c}.\]
        Additionally, since $xy_+=y_+$ and $xy_-=-y_-$ we must have
            \[\kappa_{\phi^{-1}(b_0)}l^+_{c,c}= l^+_{\phi^{-1}(c),c}\text{ and }\kappa_{\phi^{-1}(b_0)} l^-_{c,c}=-l^-_{\phi^{-1}(c),c}\]
        giving us
            \[\kappa_{\phi^{-1}(b_0)}=l^+_{\phi^{-1}(c),c}-l^-_{\phi^{-1}(c),c}=2l^+_{\phi(c),c}\ ,\]
        which is a contradiction since $\kappa_{\phi^{-1}(b_0)}\in\{\pm 1\}$ and $l^+_{\phi^{-1}(c),c}\in\mathds{Z}$.
        
        So we must have $\phi(b)=b$ for each $b\in B$ and therefore we have
        \[x=\sum_{b\in B} \kappa_b bb^*.\]
    \end{proof}
    Since the Higman--Thompson groups of unfolding trees are isomorphic to \[U\big(L_{\mathds{Z}}(G,S)\big)/{DU\big(L_{\mathds{Z}}(G,S)\big)}\] and the groups in this factor are both definable via a first order logical sentence (in the language of $*$-algebras) in $L_{\mathds{Z}}(G,S)$ we can see that whenever the rooted Leavitt path algebras are isomorphic, so are the Higman--Thompson groups. Additionally, since any isomorphism between two Leavitt path modules $L_{\mathds{Z}}(G)$ and $L_{\mathds{Z}}(H)$, that sends a root $R$ of $G$ to a root $S$ of $G$ also sends the rooted Leavitt path subalgebra $L_{\mathds{Z}}(G,R)$ to $L_{\mathds{Z}}(H,S)$, thus inducing an isomorphism between the Higman--Thompson groups. Note also that an isomorphism of unitary subgroups of these algebras alone does not necessarily induce an isomorphism of Higman--Thompson groups in the same way, since the definition of $DU\big(L_{\mathds{Z}}(G,R)\big)$ quantifies over all of $L_{\mathds{Z}}(G,R)$ and uses addition.
    \begin{theorem}\label{ch4-thm2}
        For any rooted graphs, $(G,R),(H,S)$ if there exists an $*$-isomorphism $\phi:L_{\mathds{Z}}(G,R)\to L_{\mathds{Z}}(H,S)$ then
            \[\mathcal{HT}\big(\mathcal{T}(G,S)\big)\cong\mathcal{HT}\big(\mathcal{T}(H,R)\big).\]
    \end{theorem}
    \begin{proof}
        Since $\phi$ is an $*$-isomorphism, we have
        %     \[L_{\mathds{Z}}(G,R)=\{ x\in L_{\mathds{Z}}(G)\mid Rx=xR=x \}\]
        % we must have $\phi(L_{\mathds{Z}}(G,R))=L_{\mathds{Z}}(H,S)$ and thus 
        \[\phi\Big(U\big((L_{\mathds{Z}}(G,R)\big)\Big)=U\big((L_{\mathds{Z}}(H,S)\big)\] and by \lemref{ch4-lem10} \[\phi(DU((L_{\mathds{Z}}(G,R)))=DU((L_{\mathds{Z}}(H,S)).\] So by restricting and filtering through $DU\big((L_{\mathds{Z}}(G,R)\big)$ we get an isomorphism
            \[\bar{\phi}:U\big((L_{\mathds{Z}}(G,R)\big)/{DU\big((L_{\mathds{Z}}(G,R)\big)}\to U\big((L_{\mathds{Z}}(H,S)\big)/{DU\big((L_{\mathds{Z}}(H,S)\big)},\]
        which is the required isomorphism.

    \end{proof}

    We can combine our above theorem with results from \cite{ABRAMS2008}  and provide another reduction of a graph that preserves the Higman--Thompson group of its unfolding tree. 

    \begin{defn}[\cite{ABRAMS2008}, Definition 2.1]
        For any graph $G$ and any two vertices $v,w\in VG$ s.t. there exists a injection
            \[\theta:o^{-1}(w)\to o^{-1}(v) \]
        with $t(\theta(e))=t(e)$. The graph $G(w\hookrightarrow v)$ is defined by
            \begin{align*}
                VG(v\hookrightarrow w):&=VG\\
                EG(v\hookrightarrow w):&=EG\setminus \theta(o^{-1}(w))\cup \{f_{v,w}\} 
            \end{align*}
        with $o(f_{v,w})=v$ and $t(f_{v,w})=w$.
    \end{defn}

    \begin{cor}
        For any rooted graph $(G,R)$ and any $v,w\in VG$, s.t. there is an injection
            \[\theta:o^{-1}(w)\to o^{-1}(v) \]
        with $t(\theta(e))=t(e)$, we have 
            \[\mathcal{HT}\big(\mathcal{T}(G,R)\big)\cong \mathcal{HT}\Big(\mathcal{T}\big(G(w\hookrightarrow v),R\big)\Big)\]
    \end{cor}
    \begin{proof}
        We will define for any $e\in o^{-1}(w)$
            \[r_e:=f_{v,w}e\in L_{\mathds{Q}}\big(G(w\hookrightarrow v)\big).\]
        By \cite[Theorem 2.3.]{ABRAMS2008} we have a linear $*$-isomorphism $\varphi: L_{\mathds{Q}}(G)\to L_{\mathds{Q}}\big(G(w\hookrightarrow v)\big)$, with 
        \begin{itemize}
            \item \(\forall v\in VG,\ \varphi(v)=v,\)
            \item \(\forall e\in EG\setminus\theta\big(o^{-1}(w)\big),\ \varphi(e)=e\)\text{ and }
            \item  \(\forall e\in o^{-1}(w),\ \varphi\big(\theta(e)\big)=r_e.\)
        \end{itemize}
        Since $\varphi(R)=R$ we must have $\varphi\big(L_{\mathds{Q}}(G,R)\big)=L_{\mathds{Q}}\big(G(w\hookrightarrow v),R\big)$. Additionally, since $\varphi$ is $\mathds{Q}$-linear and for any paths $p,q$ in $G$ $\varphi(pq^*)$ is of the form $rt^*$ for some paths $r,t$ in $G(w\hookrightarrow v)$, we must have $\varphi\big(L_{\mathds{Z}}(G)\big)=L_{\mathds{Z}}\big(G(w\hookrightarrow v)\big)$.

        Combining these two observations gives us that the restriction $\varphi|_{L_{\mathds{Z}}(G,R)}$ gives us an isomorphism from $L_{\mathds{Z}}(G,R)$ to $L_{\mathds{Z}}\big(G(w\hookrightarrow v),R\big)$. So we can apply \thmref{ch4-thm2} and get the required isomorphism. 
    \end{proof}

    As another application of \lemref{ch4-thm2} we can define a simple way of modifying the graph, that preserve the rooted Leavitt path algebra and thus the Higman--Thompson groups.
    \begin{defn}
        For any rooted graph $(G,R)$ and any vertex $R\neq v\in VG$, with $o^{-1}(v)\cap t^{-1}(v)=\emptyset$ (i.e. no loops based on $v$), we define the graph $G_v$ with vanished vertex $v$ as follows:
        \begin{align*}
            &VG_v:= VG\setminus\{v\}\\
            &EG_v:= \{e\in EG\mid t(e),o(e)\neq v\}\cup\{(e_1,e_2)\in EG^2\mid t(e_1)=o(e_2)=v\}\\
            &\forall (e_1,e_2)\in EG_v\cap EG^2\ o_{G_v}\big((e_1,e_2)\big):=o_G(e_1),\ t_{G_v}\big((e_1,e_2)\big):=t_G(e_2)\\
            &\forall e\in EG_v\cap EG\ o_{G_v}(e):=o_G(e),\ t_{G_v}(e):=t_G(e)\\
        \end{align*}
    \end{defn}
    \begin{figure}
        \begin{subfigure}[b]{0.5\textwidth}
            \begin{tikzpicture}
          % nodes
            \node[big red node]            (v0)  at (0,3)              {};
            \node[big green node]                 (vxx) at (-1,1)             {};
            \node[big blue node]              (vxy) at (1,1)              {};
            %\node[square red node]                (vyx) at (-1,-1)             {};
            \node[big yellow node]             (vyy) at (1,-1)               {};
            % edges
            \path[-latex]
                %   FROM            BEND/LOOP                                   POSITION OF LABEL       LABEL           TO
                    (vxx)           edge[loop, out= 90, in= 180, looseness=15]  node[right]             {}              (vxx)
                    (vxx)           edge[bend left]                             node[above]             {}           (vxy)
                    (vxy)           edge[bend left]                             node[right]             {}              (vyy)
                    (vyy)           edge[bend left]                             node[below]             {}           (vxx)
                    (vyy)           edge[bend left]                             node[left]              {}           (vxy)
                    (v0)            edge[]                                      node[left]              {}           (vxx)
                    (v0)            edge[]                                      node[right]             {}           (vxy)
                    
                    ;
        \end{tikzpicture}
        \vspace{1cm}
        \end{subfigure} 
        \begin{subfigure}[b]{0.5\textwidth}
            \begin{tikzpicture}
          % nodes
            \node[big red node]            (v0)  at (0,3)              {};
            \node[big green node]                 (vxx) at (-1,1)             {};
            %\node[big yellow node]              (vxy) at (1,1)              {};
            %\node[square red node]                (vyx) at (-1,-1)             {};
            \node[big yellow node]             (vyy) at (1,-1)               {};
            % edges
            \path[-latex]
                %   FROM            BEND/LOOP                                   POSITION OF LABEL       LABEL           TO
                    (vxx)           edge[loop, out= 90, in= 180, looseness=15]  node[above]             {}           (vxx)
                    (vxx)           edge[bend left]                             node[above]             {}           (vyy)
                    (vyy)           edge[bend left]                             node[below]             {}          (vxx)
                    (vyy)           edge[loop, out= 270, in= 360, looseness=15] node[below]             {}           (vyy)
                    (v0)            edge[]                                      node[left]              {}           (vxx)
                    (v0)            edge[bend left]                             node[right]             {}           (vyy)
                    ;
        \end{tikzpicture}
        \end{subfigure}
        \caption{Vanishing the blue vertex}
        \label{fig:placeholder}
    \end{figure}
    We will call going from $G$ to $G_v$: \textbf{vanishing a vertex}. The inverse of vanishing will be called \textbf{pinching}. We can show that this move preserves the rooted Leavitt path algebra (although not necessarily the unrooted Leavitt path algebra).
    \begin{lemma}\label{ch4-lem11}
        For any ring $\mathcal{R}$, any rooted graph $(G,R)$ and any vertex $R\neq v\in VG$, with $o^{-1}(v)\cap t^{-1}(v)=\emptyset$, we have 
            \[L_{\mathcal{R}}(G,R)\cong L_{\mathcal{R}}(G_v,R).\]
    \end{lemma}
    \begin{proof}
        We can define for every path $p=e_1\dots e_n\in\mathcal{W}(G,R)$ such that $T(p)\neq v$ the path $p_v\in \mathcal{W}(G_v,R)$ recursively over the length of $p$ as follows. Firstly $(\varepsilon_R)_v:=\varepsilon_R$. If $o(e_{n})\neq v$ we simply set $p_v=(e_1\dots e_{n-1})_ve_n$. If $o(e_{n})=v$ we have some path $q\in\mathcal{W}(G,R)$ s.t. $p=qe_{n-1}e_n$. Note that since $t(e_{n-1})=o(e_n)=v$, we must have $T(q)=o(e_{n-1})\neq v$ and thus $q_v$ is defined by the prior consideration. We can therefore define $p_v:=q_v(e_{n-1},e_{n})$. 

        Additionally, we note that $\{pq^*\mid T(p)=T(q)\neq v\}$ spans $L_{\mathcal{R}}(G,R)$ since for any $p,q\in \mathcal{W}(G,R)$ with $T(p)=T(q)=v$ we have
            \[pq^*=\sum_{e\in o^{-1}(v)}pe(qe)^*\]
        and for any $e\in o^{-1}(v)$, $T(pe)=T(qe)=t(e)\neq v$.
        So the assignment $pq^*\mapsto p_vq_v^*$ for any paths $p,q$ that don't end in $v$, induces a function $\phi:L_{\mathcal{R}}(G,R)\to L_{\mathcal{R}}(G_v,R)$. By straightforward checking we can see that it is in fact a $*$-homomorphism. $\phi$ is surjective since for any $r\in\mathcal{W}(G_v,R)$ there exist some $p\in\mathcal{W}(G,R)$ such that $p_v=r$. To see that the function is injective we first note that $p\neq q$ implies $p_v\neq q_v$. Therefore if we have a finite maximal basis $B$ and a finite set $M$ with $v\notin T(B)\cup T(M)$ and values $k_{p,q}\in\mathcal{R}$ for any pair $(p,q)\in B\times M$ with $T(p)=T(q)$ such that 
            \[0=\phi\Big(\sum_{\substack{p\in B,q\in M\\ T(p)=T(q)}}k_{p,q} pq^*\Big)=\sum_{\substack{p\in B,q\in M\\ T(p)=T(q)}}k_{p,q} p_vq_v^*,\]
        we must have $k_{p,q}=0$ for any $p\in B,q\in M$, as $B_v:=\{p_v\mid p\in B\}$ is also a basis.
    \end{proof}
    Immediately we therefore have the following.
    \begin{theorem}
        For any rooted graph $(G,R)$ and any $R\neq v\in VG$ with $o^{-1}(v)\cap t^{-1}(v)=\emptyset$, we have
            \[\mathcal{HT}\big(\mathcal{T}(G,R)\big)\cong \mathcal{HT}\big(\mathcal{T}(G_v,R)\big).\]
    \end{theorem}
    \begin{proof}
        This follows from combining \thmref{ch4-thm2} and \lemref{ch4-lem11}.
    \end{proof}
    Note that each of the graph moves mentioned in \cite[Appendix 3]{abrams2015leavitt} that preserve the Leavitt path algebra can be achieved by pinching and vanishing. This naturally begs whether all isomorphisms between rooted Leavitt path algebras can be realized with these moves.  
    \section{Acknowledgments}
        I would like to thank my supervisors: George Willis, Stephan Tornier and  Colin Reid for their advice and review while writing this paper. I would also like to thank the ARC for funding the research that led to this paper via the grant no. FL170100032. This paper was adapted from the last chapter of my thesis that has been submitted for review.
    \bibliographystyle{plain}
    % or: plain,unsrt,alpha,abbrv,acm,apalike,...
    \bibliography{bibliography}
    \end{document}